\documentclass[10pt]{amsart}
\usepackage{enumerate,amsmath,amssymb,latexsym,
amsfonts, amsthm, amscd, MnSymbol}

\theoremstyle{plain}

\newtheorem{theorem}{Theorem}

\numberwithin{equation}{section}

\newcommand{\LL}{\mathbb{L}}
 
\newcommand{\mtd}{\medtriangledown}
\newcommand{\dagg}{\; \dagger \:}

\addtocounter{section}{-1}

\begin{document}

\title {Projective Space: Tetrads and Harmonicity}

\date{}

\author[P.L. Robinson]{P.L. Robinson}

\address{Department of Mathematics \\ University of Florida \\ Gainesville FL 32611  USA }

\email[]{paulr@ufl.edu}

 \keywords{}

\begin{abstract}

Within an axiomatic framework for three-dimensional projective space based on lines alone, we explore the Fano axiom of harmonicity according to which the diagonal lines of a complete quadrilateral are not concurrent. 

\end{abstract}

\maketitle

\medbreak

\section*{Introduction} 

In the early history of projective geometry, it was taken for granted that the diagonal points of a complete quadrangle were not collinear; indeed they were not, in the projective spaces that were then under consideration. Once proper attention was paid to the foundations of projective geometry, it was recognized that the noncollinearity of these diagonal points had to be explicitly taken as axiomatic. Thus, in the classic treatise [3] of Veblen and Young there appears on page 45 the following `Fano' axiom. 
\medbreak 
ASSUMPTION H$_0$. The diagonal points of a complete quadrangle are not collinear. 
\medbreak 
Here, the indefinite article may be understood in either the universal sense or the existential sense: in context, if one complete quadrangle has noncollinear diagonal points, then so have all. The main purpose of Assumption H$_0$ is to support the important notion of harmonic set; under the contrary assumption, harmonic sets collapse  and any element of a one-dimensional primitive form is then its own harmonic conjugate with respect to any two other elements. Accordingly, we may refer to Assumption H$_0$ as the axiom of {\it harmonicity}. 

\medbreak 

In [1] we presented a purely `linear' axiomatic framework for three-dimensional projective space, founded only on the notions of line and abstract incidence, with point and plane as derived notions; there we showed that our `linear' axiomatic framework is equivalent to the classical Veblen-Young framework with its axioms of alignment and extension. In [2] we explored the Veblen-Young axiom of projectivity within our `linear' axiomatic framework, the appropriate version of this axiom being phrased in terms of reguli. Here we explore the appropriate axiom of harmonicity within our `linear' axiomatic framework, demonstrating that its inclusion as an axiom does not disturb the crucial principle of duality. 

\medbreak 

\section*{Tetrads} 

\medbreak 

We begin by reviewing briefly the axiomatic framework of [1] essentially as recalled in [2]. The set $\LL$ of {\it lines} is provided with a symmetric reflexive relation $\dagg$ of {\it incidence} that satisfies AXIOM [1] - AXIOM [4] below. For convenience, when $S \subseteq \LL$ we write $S^{\dagg}$ for the set comprising all lines that are incident to each line in $S$; in case $S = \{ l_1, \dots , l_n \}$ we write $S^{\dagg} = [ l_1 \dots  l_n ]$. Further, when the lines $a, b \in \LL$ are not incident we call them {\it skew}.

\medbreak 

AXIOM [1]: For each line $l$ the set $l^{\dagg}$ contains three pairwise skew lines. 

\medbreak 

AXIOM [2]: For each incident pair of distinct lines $a, b:$ \par
\indent \indent \indent \indent \indent [2.1]  the set $[a b]$ contains skew pairs of lines; \par 
\indent \indent \indent \indent \indent [2.2]  if $c \in [a b] \setminus [a b]^{\dagg}$ is one of such a skew pair then no skew pairs lie in $[a b c]$; \par 
\indent \indent \indent \indent \indent [2.3] if $x, y$ is a skew pair in $[ab]$ then $[a b] = [a b x] \cup [a b y].$

\medbreak 

AXIOM [3]: If $a, b$ is an incident line pair and $c \in [a b] \setminus [a b]^{\dagg}$ then there exist an incident line pair $p, q$ and an $r \in [p q] \setminus [p q]^{\dagg}$ such that $[a b c] \cap [p q r] = \emptyset.$

\medbreak 

AXIOM [4]: Whenever $a, b$ and $p, q$ are pairs of distinct incident lines, 
$$(a \upY b) \cap (p \upY q) \neq \emptyset, \; \; \; 
(a \mtd b) \cap (p \mtd q) \neq \emptyset.$$

\medbreak 

Regarding this last axiom, we remark that on the set $\Sigma (a, b) = [a b] \setminus [a b]^{\dagg}$ (comprising all lines that are one of a skew pair in $[a b]$) incidence restricts to an equivalence relation having two equivalence classes, which we denote by $\Sigma_{\upY} (a, b)$ and $\Sigma_{\mtd} (a, b)$: the {\it point} $a \upY b = [a b c_{\upY} ]$ does not depend on the choice of $c_{\upY} \in \Sigma_{\upY} (a, b)$; likewise, the {\it plane} $a \mtd b = [a b c_{\mtd}]$ is independent of $c_{\mtd} \in \Sigma_{\mtd} (a, b)$. 

\medbreak 

By a {\it triad} we mean three pairwise-incident lines $a, b, c$ that satisfy any (hence each) of the equivalent conditions
$$a \in \Sigma (b, c), \; b \in \Sigma (c, a), \; c \in \Sigma (a, b).$$
Triads come in precisely two types: $\mtd$-triads, for which 
$$a \in \Sigma_{\mtd} (b, c), \; b \in \Sigma_{\mtd} (c, a), \; c \in \Sigma_{\mtd} (a, b);$$
$\upY$-triads, for which
$$a \in \Sigma_{\upY} (b, c), \; b \in \Sigma_{\upY} (c, a), \; c \in \Sigma_{\upY} (a, b).$$

\medbreak 

We shall frequently use the terminology of classical projective geometry. For instance, we may say that a point and a plane are {\it incident} when their intersection is nonempty; more generally, a collection of points and planes will be called {\it collinear} when they contain a common line. Also, lines will be called {\it coplanar} when they all lie in one plane: that is, when they are all elements of $a \mtd b$ for some incident $a \ne b \in \mathbb{L}$. Dually, lines will be called {\it copunctal} (or concurrent) when they all pass through one point: that is, when they are all elements of $a \upY b$ for some incident $a \ne b \in \mathbb{L}$. Thus, the lines of a $\mtd$-triad are coplanar but not copunctal, those of a $\upY$-triad copunctal but not coplanar. 

\medbreak 

After this brief review, we now turn our attention to sets of four lines, assumed pairwise-incident. It is of course possible that four pairwise-incident lines $o, p, q, r$ include no triad: in such a case, $o \in (p \upY q) \cap (p \mtd q) \ni r$ so that $o, p, q, r$ lie in the flat pencil comprising all lines through the point $p \upY q$ in the plane $p \mtd q$. Our interest lies in the complementary case that a triad is included. 

\medbreak 

\begin{theorem} \label{pretetrad}
If the pairwise-incident lines $o, p, q, r$ include a $\mtd$-triad, then they are coplanar and do not include a $\upY$-triad. 
\end{theorem} 

\begin{proof} 
Let $p, q, r$ be a $\mtd$-triad: thus, $r \in \Sigma_{\mtd} (p, q)$ and $q \mtd r = r \mtd p = p \mtd q$. It cannot be that $o \in \Sigma_{\upY} (p, q)$ for that would preclude $o \dagg r$, so $o, p, q$ is not a $\upY$-triad; consequently, $o \in [p q] \setminus \Sigma_{\upY} (p, q) = p \mtd q$. 
\end{proof} 

\medbreak

\medbreak 

Thus, the six planes $o \mtd p, \; o \mtd q, \;  o \mtd r, \; q \mtd r, \; r \mtd p, \; p \mtd q$ coincide; also, if $o, p, q, r$ includes another triad then it must be of the $\mtd$ variety. 

\medbreak 

Dually, if four pairwise-incident lines include a $\upY$-triad then they are copunctal and do not include a $\mtd$-triad. 

\medbreak 

We define a {\it tetrad} to be a set of four (pairwise-incident) lines $\{ o, p, q, r \}$ such that each of the four included triples $\{p, q, r \}, \; \{ o, q, r \}, \; \{o, r, p \}, \; \{o, p, q \}$ is a triad. 

\medbreak 

Theorem \ref{pretetrad} implies that, like triads, tetrads come in precisely two types: $\mtd$-tetrads, in which each of the included triples is a $\mtd$-triad; $\upY$-tetrads, in which each of the included triples is a $\upY$-triad. Plainly, $\mtd$-tetrads and $\upY$-tetrads are dual notions. 

\medbreak 

Let $o, p, q, r$ be a $\mtd$-tetrad. Note at once that the points $o \upY r$ and $p \upY q$ are distinct: indeed, $r \in \Sigma_{\mtd} (p, q) = [p q] \setminus (p \upY q)$; likewise, $o \upY q \ne r \upY p$ and $o \upY r \ne p \upY q$. Accordingly, the (coplanar) {\it diagonals} appearing in the next result are well-defined. In the statement, we have taken a harmless liberty: for example, $c = (o \upY r) \cap (p \upY q)$ should strictly read $\{c\} = (o \upY r) \cap (p \upY q)$. 

\medbreak 

\begin{theorem} \label{diagonals}
If $o, p, q, r$ is a $\mtd$-tetrad then its diagonals 
$$a = (o \upY p) \cap (q \upY r), \; \; b = (o \upY q) \cap (r \upY p), \; \; c = (o \upY r) \cap (p \upY q)$$
are distinct from each other and from $o, p, q, r$. 
\end{theorem} 

\begin{proof} 
Note first from $o \notin p \upY q \ni c$ that $o \ne c$; entirely similar arguments show that each of $a, b, c$ differs from each of $o, p, q, r$. Now the points $o \upY c$ and $o \upY r$ share two lines and hence coincide, while $o \upY a = o \upY p$ and $o \upY b = o \upY q$ likewise: as $o, p, q$ are not copunctal, $o \upY p \ne o \upY q$ so that $o \upY a \ne o \upY b$ and therefore $a \ne b$; entirely similar arguments show that $b \ne c$ and $c \ne a$. 
\end{proof} 

\medbreak 

Dually, if $o, p, q, r$ is a $\upY$-tetrad then its (well-defined, copunctal)  diagonals
$$a = (o \mtd p) \cap (q \mtd r), \; \; b = (o \mtd q) \cap (r \mtd p), \; \; c = (o \mtd r) \cap (p \mtd q) $$ 
are distinct from each other and from $o, p, q, r$. 

\medbreak 

The question whether more can be said about the diagonals of a tetrad is addressed in the next section. 

\medbreak 

\section*{Harmonicity} 

\medbreak 

As our axiomatic framework for projective space is based on lines it is more natural to begin rather with a complete quadrilateral and its diagonal lines than with a complete quadrangle and its diagonal points. Thus, in classical projective geometric terms, we prefer to take as our version of the harmonicity axiom the planar dual of Assumption H$_0$: that is, we prefer to assume that the diagonal lines of a complete quadrilateral are not concurrent. Incidentally, we remark that Veblen and Young appear not to state explicitly that this is a consequence of their Assumption H$_0$ and hence that the principle of duality continues to hold after H$_0$ is introduced, though it is implicit in their discussion of the quadrangle-quadrilateral configuration on pages 44-46 of [3].

\medbreak 

Let us examine this version of the harmonicity axiom a little more closely. Recall that a complete quadrilateral is a planar figure comprising four lines, no three of which are concurrent, along with their three diagonal lines. The requirements that the four lines $o, p, q, r$ be coplanar and that no three of them be concurrent may be expressed as follows: if $x, y, z$ are any three of the lines, then $z \in x \mtd y$ (for coplanarity) and $z \notin x \upY y$ (for non-concurrence) so that $z \in \Sigma_{\mtd} (x, y)$. Thus: a complete quadrilateral is precisely a $\mtd$-tetrad $o, p, q, r$ along with its diagonals $a = (o \upY p) \cap (q \upY r), \; b = (o \upY q) \cap (r \upY p), \; c = (o \upY r) \cap (p \upY q)$ as defined previously. In the same spirit, the non-concurrence of $a, b, c$ corresponds to the statement that $a, b, c$ is a $\mtd$-triad. 

\medbreak 

This suggests that we adopt as our axiom of harmonicity the following statement. 

\medbreak 

{\bf [H$_{\mtd}$]} The diagonals of a $\mtd$-tetrad form a $\mtd$-triad. 

\medbreak 

It is natural to consider also the following (spatially) dual statement. 

\medbreak 

{\bf [H$_{\upY}$]} The diagonals of a $\upY$-tetrad form a $\upY$-triad. 

\medbreak 

The following is a theorem in our `linear' axiomatic framework. In the proof, we shall feel free to use classical notions and  notation when convenient, as in [1] and [2]: for example, Theorem 15 of [1] shows that if distinct points $A$ and $B$ lie in a plane then so does the unique line $A B$ that joins them, while Theorem 2 of [2] shows that if the plane $\pi$ does not contain the line $\ell$ then both pass through a unique common point $\pi \cdot \ell$. 

\begin{theorem} 
 {\bf [H$_{\mtd}$]}$\; \Rightarrow \; ${\bf [H$_{\upY}$]}.
\end{theorem} 

\begin{proof} 
Assume {\bf [H$_{\mtd}$]}. Let $o, p, q, r$ be a $\upY$-tetrad with diagonals $a, b, c$: all seven lines pass through a common point $Z = p \upY q$; no three of the lines $o, p, q, r$ are coplanar. Choose a plane $\zeta$ that does not pass through $Z$ and denote the points in which $\zeta$ meets each of the seven lines by the corresponding upper case letter: thus, $O = \zeta \cdot o , \dots , C = \zeta \cdot c$. Now, consider the quadrilateral $A P B Q$ in the plane $\zeta$: its four sides are $A P, P B, B Q, Q A$ and its three diagonals are $P Q, R O, A B$. 

\medbreak\noindent
{\it Claim}:  the four lines $A P, P B, B Q, Q A$ constitute a $\mtd$-tetrad. \par
[As the four lines are coplanar, we need only show that no three are concurrent. By symmetry, we need only show that neither $B$ nor $Q$ lies on the line 
$$A P = \zeta \cap (a \mtd p ) = \zeta \cap (o \mtd p).$$
Suppose that $B$ lies on $A P$: then $o \mtd p$ passes through both $Z$ and $B$ and hence contains $Z B = b = (o \mtd q) \cap (r \mtd p)$; now $\{ o, b \} \subseteq (o \mtd p) \cap (o \mtd q)$ so that $o \mtd p = o \mtd q$ and $o, p, q$ are coplanar, contrary to hypothesis. The supposition that $Q$ lies on $A P$ immediately places $q$ in $o \mtd p$ and again contradicts the non-coplanarity of $o, p, q$.]

\medbreak

Now {\bf [H$_{\mtd}$]} renders the diagonals of the quadrilateral $A P B Q$ non-concurrent: that is, the lines
$$P Q = \zeta \cap (p \mtd q), \; R O = \zeta \cap (r \mtd o), \; A B = \zeta \cap (a \mtd b)$$
have no point in common. It follows that the planes $p \mtd q, \; r\mtd o, \; a \mtd b$ are not collinear: indeed, if the line $\ell$ were common to each of these planes then the point $\zeta \cdot \ell$ would lie on each of the lines $P Q, R O, A B$. As $(p \mtd q) \cap (r \mtd o) = \{c\}$ we deduce that $c \notin a \mtd b$ and conclude that $c \in \Sigma_{\upY} (a, b)$: so $a, b, c$ is a $\upY$-triad. 
\end{proof} 

\medbreak 

The principle of duality holding in our axiomatic framework, it follows that the dual of {\bf [H$_{\mtd}$]}$\; \Rightarrow \; ${\bf [H$_{\upY}$]} is also a theorem; of course, this dual is precisely the converse {\bf [H$_{\upY}$]}$\; \Rightarrow \; ${\bf [H$_{\mtd}$]}. Consequently, the dual statements {\bf [H$_{\mtd}$]} and {\bf [H$_{\upY}$]} are equivalent: if we adopt the one as an additional axiom then the other becomes a theorem and the crucial principle of duality is preserved. Alternatively, we may combine the two and adopt as our harmonicity axiom the following (equivalent) self-dual version. 

\medbreak 

AXIOM {\bf [H]}: The diagonals of a tetrad form a triad of like type.

\medbreak 

We now come full circle: on the basis of {\bf [H]} we establish as a theorem the assertion with which we opened this paper. 

\medbreak 

\begin{theorem} 
The diagonal points of a complete quadrangle are not collinear. 
\end{theorem} 

\begin{proof} 
Let $O, P, Q, R$ be the vertices and $A = OP.QR, \; B = OQ.RP, \; C = OR.PQ$ the diagonal points of a complete quadrangle. Fix a point $Z$ off the plane of the quadrangle and join $Z$ to each of the seven other points, thereby obtaining $o = ZO, \dots , c = Z C$. The lines $o, p, q, r$ form a $\upY$-tetrad with the lines $a, b, c$ as its diagonals. Axiom {\bf [H]} declares that $a, b, c$ form a $\upY$-triad and are therefore not coplanar, whence $A, B, C$ are not collinear. 
\end{proof} 

\medbreak 

\bigbreak

\begin{center} 
{\small R}{\footnotesize EFERENCES}
\end{center} 
\medbreak 

[1] P.L. Robinson, {\it Projective Space: Lines and Duality}, arXiv 1506.06051 (2015). 

\medbreak 

[2] P.L. Robinson, {\it Projective Space: Reguli and Projectivity}, arXiv 1506.08217 (2015). 

\medbreak 

[3] O. Veblen and J.W. Young, {\it Projective Geometry}, Volume I, Ginn and Company, Boston (1910).

\medbreak

\end{document}